%% file: div-free_3d-2d.tex
\documentclass[draft,12pt]{article}

\usepackage{a4}
\usepackage[english]{babel}                 
\usepackage{amsmath}                       
\usepackage{amssymb}                       
\usepackage{amsfonts}                      
\usepackage{enumerate}                     
\usepackage{amsthm}                        
\usepackage{hyperref}
\usepackage{epsf}
\usepackage[T1]{fontenc}
\usepackage[ansinew]{inputenc}

\input{10macros}

\author{
Stefan Krömer\\
\small
Universität zu Köln,\\
\small
skroemer@math.uni-koeln.de
}
\date{20 April 2010
}
\title{Dimension reduction for functionals on solenoidal vector fields
}%

\begin{document}
\maketitle
\begin{abstract}
We study integral functionals constrained to divergence-free vector fields in $L^p$ on a thin domain, 
under standard $p$-growth and coercivity assumptions, $1<p<\infty$. We prove that as the thickness of the domain goes to zero, the Gamma-limit 
with respect to weak convergence in $L^p$ is always given by the associated functional with convexified energy density wherever it is finite.
Remarkably, this happens despite the fact that relaxation of nonconvex functionals subject to the limiting constraint can give rise to a nonlocal functional as illustrated
in an example.\\[1ex]
\end{abstract}

\section{Introduction}\label{sec:intro}

This article is devoted to the study the ``effective'' value per unit volume of functionals
constrained to solenoidal (i.e., divergence-free) vector 
fields defined on a thin domain $\omega\times (0,\eps)$, in the limit as the thickness $\eps$ goes to zero.
We assume that on a domain with finite thickness, our functional (which we call the ``energy'', although its meaning might be different from a physical point of view) 
is given by 
a integral of the form
\begin{align*}
		G_\eps(v):=\left\{\begin{alignedat}{2}
		&\frac{1}{\eps}\int_{\omega\times (0,\eps)}
		g\big(y',v(y)\big)
		\,dy~~&&\text{if}~v\in \cV_\eps,\\
		&+\infty~~&&\text{if}~v\in L^p(\omega\times (0,\eps);\RR^N)\setminus \cV_\eps
		\end{alignedat}\right.
\end{align*}
where $N\geq 3$, $\omega$ is a bounded domain in $\RR^{N-1}$, $y=(y',y_N)\in \omega\times (0,\eps)$, $g:\omega\times \RR^N\to \RR$ is a given energy density, and 
$G_\eps$ is finite only in the class of solenoidal vector fields on $\omega\times (0,\eps)$ in $L^p$ for some $1<p<\infty$, i.e.,
\begin{align*}
	\cV_\eps:=\left\{v\in L^p(\omega\times (0,\eps);\RR^3) \,\left|\, 
	\div v=0 \right.\right\}.
\end{align*}
Here and throughout the rest of this article, differential constraints as for $v$ above are
understood in the sense of distributions, in particular, $\div v=0$ for a $v\in L^p(\omega\times (0,\eps);\RR^3)$ means that $\int_{\omega\times (0,\eps)} v\cdot \nabla 
\varphi\,dy=0$ for all test functions $\varphi\in C_c^\infty(\omega\times (0,\eps))$ (smooth functions with compact support, scalar-valued).
Using rescaled variables given by $x=(x',x_N)=(y',\eps^{-1}y_N)$ and 
$u(x)=v(x',\eps x_N)$,
$G_\eps$ is transformed into a functional defined on a fixed domain:
\begin{align*}
		F_\eps(u):=\left\{\begin{array}{ll}
		\int_{\Omega} f\big(x,u(x)\big) \,dx,~~
		&\text{if}~u\in \cU_\eps,~~\text{with}~\Omega:=\omega\times (0,1),\\
		+\infty ~~&\text{if}~u\in L^p(\Omega;\RR^N)\setminus \cU_\eps,
		\end{array}\right.
\end{align*}
where $f(x,\cdot)=g(x',\cdot)$ for $x=(x',x_N)\in \RR^{N-1}\times \RR$,
$$
  \cU_\eps:=\mysetl{u\in L^p(\Omega;\RR^3)}{
	\div_\eps u=0}
$$
and
$$
  \div_\eps u:=\div' u'+\tfrac{1}{\eps}\partial_N u^N:=
	\big(\textstyle{\sum_{\alpha=1}^{N-1}}\partial_\alpha 
	u^\alpha\big)+\frac{1}{\eps}\partial_N u^N
$$
for $u=(u',u^N)=(u^1,\ldots,u^N)$. 
As this does not further complicate our approach, we allow $f$ to explicitly depend on $x_N$ as well below.
We assume that 
\begin{align*}
	\text{$f:\Omega\times \RR^N\to \RR$ is a Carathéodory function\footnotemark}\label{f0}\tag{f:0}
\end{align*}
\footnotetext{i.e.~measurable in its first and continuous in its second variable} 
satisfying the following structural conditions:
\begin{alignat*}{4}
	&\text{(growth)}~&&
	\abs{f(x,\mu)}&&\leq C\abs{\mu}^p+C, &&
	\label{f1}\tag{f:1}\\
	&\text{(coercivity)}~&&
	~f(x,\mu)&&\geq \frac{1}{C}\abs{\mu}^p-C,~~&&
	\label{f2}\tag{f:2}
\end{alignat*}
with constants $C>0$ and $1<p<\infty$, for every $\mu\in \RR^N$ and a.e.~$x\in\Omega$.

Using the notion of $\Gamma$-convergence introduced by {\sc De Giorgi} \cite{DeGioFra75a,DeGioDalMa83a},
the effective energy of in the limit $\eps\to0^+$ is expressed by the $\Gamma$-limit of $F_\eps$.
with respect to weak convergence in $L^p$. 
For an introduction to the theory of $\Gamma$-convergence, the reader is referred to \cite{Dal93B} and \cite{Brai02a}.
We use the notation
\begin{alignat*}{2}
	&\Gamma(L^p_\text{weak})-\liminf F_\eps(u)&&:=
	\inf\bigset{\liminf_{\eps\to 0^+} F_\eps(u_\eps)}
	{u_\eps \rightharpoonup u
	~\text{weakly in}~L^p},\\
	&\Gamma(L^p_\text{weak})-\limsup F_\eps(u)&&:=
	\inf\bigset{\limsup_{\eps\to 0^+} F_\eps(u_\eps)}
	{u_\eps \rightharpoonup u
	~\text{weakly in}~L^p}.
\end{alignat*}
Below, we omit the topology indicated in brackets as throughout this paper, this is always the weak topology in $L^p$. We say that $\Gamma-\lim F_\eps$ exists if $\Gamma-\liminf F_\eps$ and $\Gamma-\limsup F_\eps$ coincide, in which case this quantity is denoted by $\Gamma-\lim F_\eps$.
In particular, the use of the weak topology in $L^p$ causes a process of relaxation in the limit, roughly speaking because
energetically favorable microstructures of a characteristic size converging to zero as $\eps\to 0$ 
are allowed along the sequences generating the effective (macroscopic) limiting energy.

The corresponding problem of dimension reduction for functionals depending on gradients instead of divergence-free fields was investigated by 
\textsc{Le Dret} and \textsc{Raoult} \cite{LeDreRa95a, LeDreRa96a, LeDreRa00a} and stimulated a great deal of further research, including the study of different scalings, partially with energy densities that are realistic from the point of view of hyperelasticity (see 
\cite{FrieJaMue06a} and the references therein), as well as
extensions to non-flat limiting surfaces \cite{LePa09a,LeMaPa10a}.


Recently, dimension reduction problems for Ginzburg-Landau-type functionals, involving a magnetic potential which is divergence-free as a choice of gauge, were studied in
\cite{CoSte10a} and \cite{AlBroGa10a}. In both cases, the relevant parts of the energy density (apart from compact perturbations) are convex and thus no relaxation occurs during the limit process, avoiding the main difficulty of our problem.
Relaxation and homogenization of functionals constrained to solenoidal matrix fields were treated in \cite{Pe97B} and \cite{AnsGarr07a} (for related results and some physical background also see
\cite{GaNe04a} and \cite{PaSmy09a}),
as well as in \cite{FoMue99a}, \cite{BraiFoLe00a} and \cite{FoKroe10a} for a more general constraint of the form $\cA u=0$. In this context, $\cA$ is a linear differential operator assumed to satisfy {\sc Murat}'s condition of constant rank \cite{Mu81a}, and apart from the examples in \cite{Ta79a}, \cite{Mue99c} and \cite{LeeMueMue08a}, very little is known if this condition is violated. 
In our framework, $div_\eps$ satisfies the condition of constant rank for each $\eps$, but the associated limiting operator $\div_0$ ($\div_0 u:=\partial_N u^N$ for
$u:\Omega\to \RR^N$) does not.
From the point of view of the theory for $\cA$-free fields developed in \cite{FoMue99a,BraiFoLe00a}, this means that important bounds for the projection operator onto $\div_\eps$-free fields and its complementary projection are not uniform in $\eps$ and projecting tends to create large errors as $\eps\to 0^+$ (cf. Remark~\ref{rem:projQ}). Hence, we can (and do) use the projection only along sequences that are asymptotically $\div_\eps$-free in a very strong sense (cf.~Lemma~\ref{lem:proj}).

As we shall see, the divergence-free dimension reduction problem with nonconvex energy density exhibits some intriguing features that do not occur in the gradient case. 
In particular, it turns out that dimension reduction and direct relaxation in the limit setting do not yield the same result in general. While the former simply leads to convexification by our main theorem stated below, the latter may give rise to a nonlocal functional as illustrated by the example discussed in Proposition~\ref{prop:ex1}.

Unless indicated otherwise, we assume throughout that
$$
	N\geq 2,~~\omega\subset \RR^{N-1}~\text{is open and bounded},~~\Omega:=\omega\times (0,1)~~\text{and}~~1<p<\infty.
$$
\begin{thm}\label{thm:main}
Suppose that \eqref{f0}--\eqref{f2} are satisfied.
Then $\Gamma-\lim_{\eps\to 0^+} F_\eps$ (with respect to weak convergence in $L^p$) exists, and  it has the representation 
\begin{align*}
	\Gamma-\lim F_\eps(u)=
	\left\{\begin{array}{ll}
	  F^{**}(u):=\int_\Omega f^{**}(x,u)\,dx &\text{if $u\in \cU_0$},\\
	  +\infty &\text{if $u\in L^p(\Omega;\RR^N)\setminus \cU_0$},	
	\end{array}\right.
\end{align*}
where for each $x$, $f^{**}(x,\cdot)$ denotes the convex envelope of $f(x,\cdot)$ and
\begin{align*}
  \cU_0:=\mysetl{u\in L^p(\Omega;\RR^N)}{
	\partial_N u^N=0~\text{in $\Omega$}}.
\end{align*}
\end{thm}
It is fairly easy to see that both $\Gamma-\limsup F_\eps(u)$ and $\Gamma-\liminf F_\eps(u)$ are finite if and only if $u\in \cU_0$ 
(Lemma~\ref{lem:constinx3} and Lemma~\ref{lem:rc1}), and
the lower bound for $\Gamma-\liminf F_\eps(u)$ is of course a simple consequence of the weak lower semicontinuity of convex functionals (Proposition~\ref{prop:lbconvex}). However, the upper bound, $\Gamma-\limsup F_\eps(u)\leq F^{**}(u)$ for $u\in\cU_0$, is far more difficult than in the gradient case. The main issue here is that \emph{a priori}, we do not know whether or not $\Gamma-\limsup F_\eps$ is a local integral functional. The usual trick for a proof of this property, based on ``localizing'' a sequence $u_\eps$ that weakly converges to zero by multiplying it with suitable smooth cut-off functions with the desired support, does not work in our setting, at least not in direction of the last variable, because the distance of the modified sequence to the set of $\div_\eps$-free fields in $L^p$ may be of an order approaching $1/\eps$ which is an error too large to handle. Indeed, our proof of the upper bound in Section~\ref{sec:ub} (culminating in Proposition~\ref{prop:ub}) does not use this kind of truncation in direction $x_N$, instead relying on a rather explicit construction of 
suitable sequences with small support in direction of $x_N$ which are asymptotically $\div_\eps$-free 
in the sense that their distance to $\cU_\eps$ with respect to the norm of $L^p$ goes to zero as $\eps\to 0^+$ (by Lemma~\ref{lem:proj}).
A prototype of this construction for a simple example is presented in Proposition~\ref{prop:ex2}.
\section{Preliminary observations}

We first observe that both $\Gamma-\limsup F_\eps(u)$ and $\Gamma-\liminf F_\eps(u)$ are finite if and only if $u\in \cU_0$.
The following simple density result turns out to be useful.
\begin{lem}\label{lem:U0density}
With respect to the strong topology in $L^p(\Omega;\RR^N)$, $\cU_0\cap C^\infty(\bar\Omega;\RR^N)$ is dense in $\cU_0$.
\end{lem}
\begin{proof}
Let $u\in\cU_0$, and extend $u=(u^1,\ldots,u^N)$ to a function in $L^p_\loc(\Omega;\RR^N)$
such that $u^j=0$ on $\RR^N\setminus \Omega$ for $j=1,\ldots,N-1$, $u^N=0$ on $\RR^N\setminus (\omega\times \RR)$ 
and $u^N(x',x_N)$ is still constant in $x_N$ for a.e.~$x'\in\omega$. 
Mollifying in the usual way yields a sequence $(u_k)_{k\in\NN}$ in $C^\infty(\RR^N;\RR^N)\cap \cU_0$ with $u_k\to u$ strongly in $L^p(\Omega;\RR^N)$.
\end{proof}

\begin{lem}\label{lem:constinx3}
Let $v_n$ be a bounded sequence in $L^p(\Omega)$ with 
$v_n \rightharpoonup v_\infty$ weakly in $L^p(\Omega)$, and suppose that
$\partial_N v_n\to 0$ in the sense of distributions.
Then $v_\infty$ is constant in $x_N$.
In particular, if $(u_\eps)\subset \cU_\eps$ and $u_\eps\rightharpoonup u$ weakly in $L^p$, then 
$u\in \cU_0$. 
\end{lem}

\begin{proof}
For every $\varphi\in C_0^\infty(\omega)$ and every $\eta\in C_0^\infty((0,1))$, we have
\begin{align*}
	0&= \lim_{n\to\infty} \int_\omega\int_{(0,1)} v_n(x',x_N) 
	\varphi(x')\dot\eta(x_N) \,dx_N dx'\\
	&= \int_\omega\int_{(0,1)} v_\infty(x',x_N) 
	\varphi(x')\dot\eta(x_N) \,dx_N dx'.
\end{align*}
In particular, since $\varphi$ was arbitrary, we get
\begin{align*}
	\int_{(0,1)} v_\infty(x',x_N) \dot\eta(x_N) \,dx_N=0
	~~\text{for a.e.~$x'\in \omega$ and every $\eta\in C_0^\infty((0,1))$},
\end{align*}
which in turn implies that 
$v_\infty(x',x_N)$ is constant in $x_N$.
\end{proof}
\begin{lem}~\label{lem:rc1}
For every $u\in \cU_0$, there exists a sequence $(u_\eps)\subset \cU_\eps$ such that $u_\eps-u\to 0$ in $L^p(\Omega;\RR^N)$.
\end{lem}
\begin{rem}
Using Lebesgue's theorem, \eqref{f0} and \eqref{f1}, we get that $\lim F_\eps(u_\eps)= \int_\Omega f(x,u)\,dx<\infty$, and thus
$\Gamma-\limsup F_\eps(u)<\infty$ for every $u\in\cU_0$.
\end{rem}
\begin{proof}[Proof of Lemma~\ref{lem:rc1}]
{\bf \underline{Step 1}:} Assume in addition that $u\in C^1(\overline{\Omega};\RR^N)$.\\
For $j=1,\ldots,N-1$ define $u_\eps^j:=u^j$, and let
\begin{align*}
	u_\eps^N(x',x_N):=u^N(x',x_N)-\eps \int_0^{x_N} \div'u(x',t)\,dt,
\end{align*}
where $\div' u=\partial_1 u^1+\ldots+\partial_{N-1} u^{N-1}$.
We thus have that $\div_\eps u_\eps=0$ and $u_\eps\to u$ strongly in $L^p$, whence $v_\eps:=u_\eps$ has the asserted properties.

{\bf \underline{Step 2}:} The general case.\\
By Lemma~\ref{lem:U0density}, there exists a sequence $(u_k)\subset C^1(\overline{\Omega};\RR^N)\cap \cU_0$ with $u_k\to u$ strongly in $L^p(\Omega;\RR^N)$ as $k\to\infty$. For each $k$ and each $\eps$, we define $u_{k,\eps}\in \cU_\eps$ as in the first step, using $u_k$ instead of $u$. Now choose $(k(\eps))_{\eps>0}$ with $k(\eps)\to \infty$ slow enough such that 
$\eps \norm{u_{k(\eps)}}_{C^1(\overline{\Omega};\RR^N)}\to 0$ as $\eps\to 0$. 
As a consequence, $u_\eps:=u_{k(\eps),\eps}$ converges to $u$ strongly in $L^p$,
and it satisfies $\div_\eps u_\eps=0$ by construction.
\end{proof}
To prove the lower bound $\Gamma-\liminf F_\eps(u)\geq F^{**}(u)$ for $u\in\cU_0$,
we first recall the well known characterization of weak lower semicontinuity of convex functionals:
\begin{thm}[see \cite{Giu03B} or \cite{FoLe07B}, e.g.]\label{thm:wlsc}
Suppose that $f$ satisfies \eqref{f0}.
Then the functional $J:L^p(\Omega,\RR^N)\to [0,\infty]$, 
$J(u):=\int_\Omega f(x,u)\,dx$, is lower semicontinuous with respect to weak convergence in $L^p$
if and only if $f(x,\cdot)$ is convex for a.e.~$x\in\Omega$.
\end{thm}
As an immediate consequence, we have
\begin{prop}[lower bound]\label{prop:lbconvex}
Suppose that the assumptions of Theorem~\ref{thm:main} hold. Then for every $u\in \cU_0$,
\begin{align*}
	\Gamma-\liminf F_\eps(u)\geq F^{**}(u).
\end{align*}
\end{prop}
For the upper bound, we have to construct a suitable sequence $(u_\eps)\subset \cU_\eps$ such that
$u_\eps\rightharpoonup u$ in $L^p$ and $F_\eps(u_\eps)\to F^{**}(u)$, starting from a given $u\in\cU_0$. 
The main problem here is the constraint $\div_\eps u_\eps=0$.
In particular, we rely on a projection onto $\div_\eps$-free fields, which is based on the following special case of the projection used in \cite{FoMue99a}.
\begin{lem}\label{lem:projQ}
Let $1<p<\infty$ and let $Q\subset \RR^N$ be an open cube. For every $\eps>0$, there exists a linear operator $\cP_\eps:L^p(Q;\RR^N)\to L^p(Q;\RR^N)$ with the following properties:
\renewcommand{\labelenumi}{(\roman{enumi})}%
\begin{myenum}
\item $\div_\eps \cP_\eps u=0$ on $\RR^N$ for every $u\in L^p(Q;\RR^N)$, where $\cP_\eps u$ is extended $Q$-periodically.
\item $\cP_\eps w=w$ for every $w\in L^p(Q;\RR^N)$ such that $\div_\eps w=0$ on $\RR^N$, where $w$ is identified with its $Q$-periodic extension to $\RR^N$.
\item $\norm{\cP_\eps u}_{L^p(Q;\RR^N)}\leq C_\eps \norm{u}_{L^p(Q;\RR^N)}$ for every $u\in L^p(Q;\RR^N)$, with a constant $C_\eps>0$ independent of $u$.
\item 
$\norm{(I-\cP_\eps)u}_{L^p(Q;\RR^N)}\leq C_\eps \nnorm{\div_\eps u}_{W^{-1,p}(Q)}$
for every $u\in L^p(Q;\RR^N)$, with a constant $C_\eps>0$ independent of $u$.
\end{myenum}
Here, on a given domain $W^{-1,p}$ denotes the dual space of $W_0^{1,p'}$ with $p'=p/(p-1)$.
\end{lem}
\begin{proof}
For $\xi=(\xi',\xi^N)\in \RR^N\setminus \{0\}$, $(\xi',\tfrac{1}{\eps}\xi)\in \RR^{1\times N}$ has full rank independent of $\xi\neq 0$, which means that for fixed $\eps$, $div_\eps$ satisfies Murat's condition of constant rank (\cite{Mu81a}). Hence, Lemma~2.14 in \cite{FoMue99a} applies with $\cA:=\div_\eps$ and $\TTT=\cP_\eps$.
\end{proof}
\begin{rem}\label{rem:projQ}
If $p=2$ (avoiding the use of general Fourier multiplier theorems), it is easy to see from the proof of Lemma~2.14 in \cite{FoMue99a} that  (iii) and (iv) actually hold with a constants independent of $\eps$. However, we do not exploit this fact, and in any case,
the factor $\frac{1}{\eps}$ hidden in the $\div_\eps$ on the right hand side of (iv) is still a major obstacle even if the constant in (iv) does not blow up as $\eps\to 0^+$.
\end{rem}
For technical reasons, it is important for us to be able to work with sequences which are not $\div_\eps$-free but can be projected to $\div_\eps$-free sequences with an error that is negligible in the limit $\eps\to 0^+$.
The following application of Lemma~\ref{lem:projQ} gives a useful sufficient criterion for sequences with this property.
\begin{lem}\label{lem:proj}
Let $\Omega\subset \RR^N$ be open and bounded, let $1<p<\infty$ and let $\eps_n\to 0^+$. Then there exists a sequence $\sigma_n\to 0^+$ such that 
the following holds: For every sequence $(u_n)\subset L^p(\Omega;\RR^N)$ with $u_n \rightharpoonup 0$ in $L^p$ and
\begin{align} \label{lproj-1}	
	\bignorm{\div_{\eps_n}u_n}_{W^{-1,p}(\Omega)}+\bignorm{\big(u_n',\tfrac{1}{\eps_n}u_n^N\big)}_{W^{-1,p}(\Omega;\RR^N)}\leq \sigma_n,
\end{align}
where $u_n':=(u_n^1,\ldots,u_n^{N-1})$, there exists a sequence $(v_n)\subset L^p(\Omega;\RR^N)$ such that $\div_{\eps_n} v_n=0$ in $\Omega$ and $u_n-v_n\to 0$ in $L^p(\Omega;\RR^N)$.
\end{lem}
\begin{proof}
For every $k\in \NN$ choose a function $\varphi_k\in C_c^\infty (\Omega;[0,1])$ such that $\varphi_k(x)=1$ for every $x\in \Omega$ with $\dist{x}{\partial\Omega}\geq \tfrac{1}{k}$. 
Moreover, choose a cube $Q$ containing $\Omega$ and a sequence $\tilde{\sigma}_n\to 0^+$ such that 
$C_{\eps_n}\tilde{\sigma}_n\to 0$ with the constants of Lemma~\ref{lem:projQ} (iv) (which also depend on $Q$).
We define
$$
	\sigma_n:=\norm{\varphi_{j(n)}}_{W^{2,\infty}(\Omega)}^{-1}\tilde{\sigma}_n~~\text{and}~~\tilde{u}_n:=\varphi_{j(n)} u_n
$$
with a sequence of integers $j(n)\to \infty$ (fast enough) such that $u_n-\tilde{u}_n\to 0$ in $L^p(\Omega;\RR^N)$.
Since
$$
	\div_{\eps_n} (\varphi_k u_n)= \varphi_k \div_{\eps_n}u_n+
	\nabla \varphi_k\cdot \big(u_n',\tfrac{1}{\eps_n}u_n^N\big),
$$
we have that
$$
	\norm{\div_{\eps_n} (\varphi_k u_n)}_{W^{-1,p}}
	\leq \norm{\varphi_k}_{W^{2,\infty}(\Omega)}\Big( \norm{\div_{\eps_n}u_n}_{W^{-1,p}}+
	 \bignorm{\big(u_n',\tfrac{1}{\eps_n}u_n^N\big)}_{W^{-1,p}}\Big).
$$
Hence, \eqref{lproj-1} implies that
$$
	C_{\eps_n} \norm{\div_{\eps_n}\tilde{u}_n}_{W^{-1,p}(\Omega)}
	=C_{\eps_n} \norm{\div_{\eps_n}\tilde{u}_n}_{W^{-1,p}(Q)}
	\leq C_{\eps_n}\tilde{\sigma}_n\to 0
$$
as $n\to\infty$.
The sequence $v_n:=\cP_n \tilde{u}_n\in L^p(Q;\RR^N)$, restricted to $\Omega$, now has the desired properties by Lemma~\ref{lem:projQ}.
\end{proof}
Applying Lemma~\ref{lem:proj} is not easy because $\sigma_n$ might converge to zero extremely fast. Nevertheless, it turns out to be possible for certain sequences
constructed below, first in Proposition~\ref{prop:ex2} for a simple example and then in Proposition~\ref{prop:ub1} as the first step in proof of the upper bound.
\section{An example and a related relaxation problem}
When studying the dimension reduction problem for functionals depending on gradients (instead of divergence-free functions), one usually
relies on a characterization of the associated relaxed functional in the limit setting, 
both as a lower semicontinuity result for the lower bound and as a first step in the construction of a sequence for the upper bound. 
In our framework, the associated relaxed functional in the limit setting corresponds to the functional $\tilde{F}_0$ introduced below.
Although $\tilde{F}_0$ does not play a role in the proof our main result, we briefly discuss it here to point out the somewhat surprising fact that $\tilde{F}_0$ does not always give the right limiting model for the divergence-free dimension reduction problem and may even be nonlocal, in sharp contrast to the gradient case. In addition, the crucial idea for the proof of the upper bound in our main result is developed in Proposition~\ref{prop:ex2} for a simple model problem.

In the following, we consider the functional
\begin{align*}
  \tilde{F}(u):=\left\{
  \begin{array}{ll}
  \int_\Omega f(x,u)\,dx~~&\text{if $u\in \cU_0$,}\\
  +\infty ~~&\text{if $u\notin \cU_0$.}
  \end{array}\right.&
\end{align*}
By definition, the relaxed functional associated to $\tilde{F}$ is given by the lower semicontinuous hull of $\tilde{F}$ with respect to weak convergence in $L^p$.
For $u\in L^p(\Omega;\RR^N)$, it can be expressed by
\begin{align}\label{F0}	
	\tilde{F}_0(u):=\Gamma-\lim \tilde{F}(u)=
	\inf\mysetl{\liminf \tilde{F}(u_n)}{u_n\rightharpoonup u~\text{weakly in $L^p$}}	 
\end{align}
Here, note that since $\tilde{F}$ does not depend on $n$, $\Gamma-\liminf \tilde{F}=\Gamma-\limsup \tilde{F}$. Moreover,
$\cU_0$ is weakly closed in $L^p$, whence $\tilde{F}_0(u)$ is finite if and only if $u\in \cU_0$. 
\begin{prop}[partial representation of $\tilde{F}_0$]\label{prop:F0reppart}
Let $f:\omega\times \RR^N\to [0,\infty)$ (identified with $f:\Omega\times\RR^N\to \RR$ constant in $x_N$) satisfy \eqref{f0}-- \eqref{f2}.
Then for every $u\in L^p(\omega;\RR^N)$ (identified with $u\in L^p(\Omega;\RR^N)$ with \emph{all} components independent of $x_N$), we have $\tilde{F}_0(u)=F^{**}(u)$, the convexified functional.
\end{prop}
\begin{proof}
Since $f\geq f^{**}$ and $F^{**}$ is weakly lower semicontinuous in 
$L^p(\Omega;\RR^N)$, it is clear that $\tilde{F}_0(u)\geq F^{**}(u)$.
On the other hand, for any $u\in L^p(\Omega;\RR^N)$ which is constant in $x_N$, we have
\begin{align*}
	\tilde{F}_0(u)&=
	\inf\mysetl{\liminf \tilde{F}(u_n)}
  	 {u_n\rightharpoonup u~\text{weakly in $L^p(\Omega;\RR^N)$, $\partial_N u^N_n=0\in\RR$}}\\
	&\leq \inf\mysetl{\liminf \tilde{F}(u_n)}
  	 {u_n\rightharpoonup u~\text{weakly in $L^p(\Omega;\RR^N)$, $\partial_N u_n=0\in\RR^N$}}\\
  &= \inf\mysetl{\liminf \int_\omega f(x',\tilde{u}_n)\,dx'}
  	 {\tilde{u}_n\rightharpoonup u~\text{weakly in $L^p(\omega;\RR^N)$}}\\	 
  &= \int_\omega f^{**}(x',u)\,dx'
  = \int_{\Omega} f^{**}(x',u)\,dx,
\end{align*}
where we used that $\int_\omega f^{**}(x',v)\,dx'$ is the weakly lower semicontinuous hull of $v\mapsto \int_{\omega}f(x',v)\,dx'$ in $L^p$.
\end{proof}
\begin{ex} 
\label{ex:nonlocal}
Let $p=6$, let $N=2$, let $f:\RR^2\to \RR$ be the three-well potential given by
\begin{align*}	
	&f(\mu):=\abs{\mu-\zeta_1}^2\abs{\mu-\zeta_2}^2\abs{\mu-\zeta_3}^2,\\
	&\text{with}~
	\zeta_1:=(0,-1),~
	\zeta_2:=(1,0),~
	\zeta_3:=(0,1),
\end{align*}
and consider the function $u_0\in \cU_0$ given by
\begin{align*}
	u_0(x_1,x_2):=\left\{\begin{array}{ll}
		(0,0)~&\text{if $x_2\in (0,{\textstyle \frac{1}{2}}]$},\\
		(1,0)~&\text{if $x_2\in ({\textstyle \frac{1}{2}},1)$}.
	\end{array}\right.~~
\end{align*}
\end{ex}
\begin{prop}[Possible nonlocal character of $\tilde{F}_0$]\label{prop:ex1}
In the situation of Example~\ref{ex:nonlocal},
we have that
\begin{align*}
	\tilde{F}_0(u_0)>0=\frac{\nabs{\omega\times (0,{\textstyle \frac{1}{2}})}}{\abs{\Omega}}\tilde{F}_0((0,0))
	+\frac{\nabs{\omega\times ({\textstyle \frac{1}{2}},1)}}{\abs{\Omega}}\tilde{F}_0((1,0)). 
\end{align*}	
In particular, $\tilde{F}_0(u)$ cannot be written in the form $\int_\Omega V(u)\,dx$ with some function 
$V:\RR^2\to \RR$, and $\tilde{F}_0(u_0)>F^{**}(u_0)$.
\end{prop}
\begin{rem}
As recently discovered in \cite{DaFoLe09ap}, the lower semicontinuous hull with respect to strong convergence in $L^2$ of certain integral functionals of the form 
$u\mapsto \int_\Omega f(u,\nabla u)\,dx$ 
can also be nonlocal, if there is a lack of coercivity with respect to the gradient variable.
\end{rem}
\begin{proof}[Proof of Proposition~\ref{prop:ex1}]
Since $f^{**}=0$ on the closed triangle formed by $\zeta_1$, $\zeta_2$ and $\zeta_3$,
$\tilde{F}_0((0,0))=\tilde{F}_0((1,0))=0$ by Proposition~\ref{prop:F0reppart}.
To prove that $\tilde{F}_0(u_0)>0$, we proceed indirectly. Suppose that $\tilde{F}_0(u_0)=0$.
By a standard diagonalization argument, we may choose a sequence $u_n\in \cU_0$ with
$u_n\rightharpoonup u_0$ weakly in $L^6(\Omega,\RR^2)$
such that $\tilde{F}_0(u_0)=\lim \tilde{F}(u_n)$.
By passing to a subsequence (not relabeled), we may assume that 
$u_n$ generates a Young measure $\nu_x$, which for a.e.~$x\in \Omega$ is a probability measure on $\RR^2$, and by the fundamental theorem for Young measures (see \cite{Ba89a}, \cite{Mue99a} or \cite{FoLe07B}, e.g.),
also exploiting that $f\geq 0$,
we get that 
\begin{align*}
	0=\tilde{F}_0(u_0)=\lim \int_\Omega f(u_n)dx \geq \int_\Omega\int_{\RR^2} f(\xi)d\nu_x(\xi)dx. 
\end{align*}	
Since $f$ vanishes only on $\{\zeta_1,\zeta_2,\zeta_3\}$, this implies that $\nu_x$ is supported in
$\{\zeta_1,\zeta_2,\zeta_3\}$ for a.e.~$x$, i.e., 
\begin{align}\label{murepres1}
	\textstyle{\nu_{x}=\sum_{j=1}^{3}\sigma_j(x)\delta_{\zeta_j}},
\end{align}
where $\delta_z$ denotes the Dirac mass concentrated at the point $z$ in $\RR^2$.
Moreover, since $\int_{\RR^2} \xi\,d\nu_x(\xi)=u_0(x)$ and $\nu_x$ is a probability measure for a.e.~$x$, 
the coefficients $\sigma_j(x)\in [0,1]$ are determined by the linear system
\begin{align*}
	\textstyle{\sum_{j=1}^{3}\sigma_j(x) \zeta_j=u_0(x)~~\text{and}~~\sum_{j=1}^{3}\sigma_j(x)=1}.
\end{align*}
One easily checks that the unique solution of this system is given by
\begin{align}
\begin{aligned}\label{murepres2}
	&\sigma_1(x)=\tfrac{1}{2},~~\sigma_2(x)=0,~~\sigma_3(x)=\tfrac{1}{2}~~
	&\text{if $x_2\leq \tfrac{1}{2}$ (i.e., $u_0(x)=(0,0)$)},\\
	&\sigma_1(x)=0,~~\sigma_2(x)=1,~~\sigma_3(x)=0~~
	&\text{if $x_2> \tfrac{1}{2}$ (i.e., $u_0(x)=(1,0)$)}.
\end{aligned}
\end{align}
In addition, the marginal of $\nu_x$ on the second coordinate axis,
\begin{align*}
	\nu_x^2(A):=\nu_x(\omega\times A)~~\text{for $A\subset (0,1)$ Borel-measurable},
\end{align*}
is the Young measure generated by $u_n^2$ and thus independent of $x_2$.
However, this contradicts \eqref{murepres1} and \eqref{murepres2}, because the latter imply that
$\nu_x^2=\sigma_1(x)\delta_{-1}+\sigma_2(x)\delta_{0}+\sigma_3(x)\delta_{1}$, and the coefficients are not constant in $x_2$ (only piecewise).
\end{proof}
The dimension reduction problem is different because the constraint $\div_\eps u_\eps=0$ is actually genuinely less restrictive than $\partial_N u^N_\eps=0$:
\begin{prop}\label{prop:ex2}
In the situation of Example~\ref{ex:nonlocal}, for every given pair of sequences $\eps_n\to 0^+$ and $\sigma_n\to 0^+$, 
there exists a bounded sequence
$(u_n)\subset L^\infty(\Omega;\RR^N)$ such that $u_n\rightharpoonup 0$ in $L^p$,
\begin{align}\label{exfconv}
	\int_\Omega f(u_n+u_0)\,dx\to \int_\Omega f^{**}(u_0)\,dx=0,
\end{align}
and 
\begin{align}\label{exdiveconv}
	\norm{\div_{\eps_n} u_n}_{W^{-1,p}(\Omega)}
	+\bignorm{\big(u_n',
	\tfrac{1}{\eps_n}u_n^N\big)}_{W^{-1,p}(\Omega;\RR^N)}
	\leq \sigma_n
\end{align}	
for every $n$. In particular,
$u_n$ can be projected onto $\cU_{\eps_n}$ with an error that goes to zero strongly in $L^p$ by Lemma~\ref{lem:proj},
and since $\div_{\eps_n} u_0=\div'u_0'=0$, this entails that $\Gamma-\liminf F_{\eps_n}(u_0)\leq 0<\tilde{F}_0(u_0)$.
\end{prop}
\begin{proof}
For each $n$ fix a function $\varphi_n\in C_c^\infty((0,1);[0,1])$
such that $\varphi_n=1$ on $[\eps_n,1-\eps_n]$, and for $k\in\NN$ let
$$
	 w_k(t)=\big(w_k^1(t),w_k^2(t)\big):=\left\{
	\begin{alignedat}[c]{2}
	 	&\zeta_3=(0,1)~~&&\text{if $0<t\leq \tfrac{1}{2k}$},\\
	 	&\zeta_1=(0,-1)~~&&\text{if $\tfrac{1}{2k}<t\leq \tfrac{1}{k}$,}
	\end{alignedat}\right.
$$
extended periodically to a function $w_k:\RR\to \RR^2$ with period $\tfrac{1}{k}$.
Note that 
\begin{align}\label{pex2-1}
	w_k\rightharpoonup \tfrac{1}{2}\zeta_3+\tfrac{1}{2}\zeta_1=(0,0)
	~~\text{weakly in $L^p(T;\RR^2)$}
\end{align}	
for any bounded open set $T\subset \RR$.
We define $v_{k,n}\in L^p(\Omega;\RR^2)$ by 
$$
	v_{k,n}(x_1,x_2):=\left\{
	\begin{alignedat}[c]{2}
		&\varphi_n(2x_2)w_k(\tfrac{1}{\eps_n}x_1)~~ &&\text{if $0< x_2< \tfrac{1}{2}$,}\\
		&0 ~~&&\text{if $\tfrac{1}{2}\leq x_2<1$.}
	\end{alignedat}\right.
$$
Observe that although $v_{k,n}$ is not continuous, its jumps do not contribute to $\div_{\eps_n} v_{k,n}$ (as a distribution), and thus the latter is actually a function with 
$$
	\div_{\eps_n} v_{k,n}(x_1,x_2)=
	\frac{2}{\eps_n}\dot{\varphi}_n(2x_2)w^2_k(\tfrac{1}{\eps_n}x_1).
$$
In particular, as $k\to \infty$ for fixed $n$,
$\div_{\eps_n} v_{k,n}\rightharpoonup 0$ weakly in $L^p(\Omega)$ as a consequence of \eqref{pex2-1}, and thus $\div_{\eps_n} v_{k,n}\to 0$
strongly in $W^{-1,p}(\Omega)$, by compact embedding.
Analogously, we get that $v_{k,n}-u_0\to 0$ in in $W^{-1,p}(\Omega;\RR^N)$ as $k\to \infty$.
Hence, we may choose $k=k(n)$ with $k(n)\to \infty$ as $n\to \infty$ fast enough such that \eqref{exdiveconv} holds for $u_n:=v_{k(n),n}$. 
Again using \eqref{pex2-1}, it is not difficult to check that $u_n\rightharpoonup 0$ weakly in $L^p$, and \eqref{exfconv} holds as well.
\end{proof}
\begin{rem} The choice of the dimension $N=2$ is not crucial for Example~\ref{ex:nonlocal}, it is just the simplest possible case. In fact, a completely analogous argument can be used for suitable potentials $f$ with $N+1$ wells in $\RR^N$ for any $N\geq 2$.
\end{rem}
\section{The upper bound\label{sec:ub}}
In this section, we provide the remaining part of the proof of Theorem~\ref{thm:main}, namely the upper bound
$$
	\Gamma-\limsup F_\eps(u)\leq F^{**}(u)~~\text{for $u\in \cU_0$},
$$
by constructing a suitable recovery sequence. 
In particular, we need some results from convex analysis:
\begin{lem}[Carathéodory's theorem, see \cite{Ro70B}, e.g.]\label{lem:Car}
Let $g:\RR^N\to [0,\infty)$ be continuous. Then for every $\xi\in \RR^N$ and every $\delta>0$,
there exists an $m\in \{0,\ldots,N\}$ and $\xi_j\in\RR^N$, $\theta_j\in (0,1]$, $j=0,\ldots,m$,
such that $\sum_j \theta_j=1$, $\xi=\sum_j \theta_j \xi_j$,
$$
	g^{**}(\xi)
	~\leq~ \textstyle{\sum_{j=0}^m} \theta_j g(\xi_j) 
	~\leq~ g^{**}(\xi)+\delta,
$$
and the vectors
$\xi_j-\xi_0$, $j=1,\ldots,m$, are linearly independent. Here, $g^{**}$ denotes the convex envelope of $g$.
\end{lem}
\begin{lem}\label{lem:Car2}
Suppose that the assumptions of Lemma~\ref{lem:Car} hold.
If, in addition, there exist constants $p>1$ and $C>0$ such that
\begin{align}\label{lC-1}
	\frac{1}{C}\abs{\mu}^p-C\leq g(\mu) \leq C\abs{\mu}^p+C~~\text{for every $\mu\in\RR^M$},
\end{align}
then the assertion of Lemma~\ref{lem:Car} stays true even for $\delta=0$, and in this case, 
\begin{align}\label{lC-2}
	\abs{\xi_j}\leq K (\abs{\xi}+1)~~\text{for}~j=0,\ldots,m,
\end{align}
where $K$ is a constant that only depends on $p$ and $C$.
\end{lem}
\begin{proof}
With some background in convex analysis, this is not hard to prove, and we just sketch some details:
It is well known that the convex envelope of $g$ can be represented as
$$
	g^{**}(\xi)=\sup\mysetl{A(\xi)}{A:\RR^N\to \RR~\text{affine and}~A\leq g},~\xi\in\RR^N.
$$ 
If $g$ is (lower semi-)continuous and has superlinear growth, the supremum is attained at a suitable affine function $A_\xi$
(see \cite{FoLe07B}, e.g.), and $A_\xi$ always touches $g$ from below at suitable points $\xi_j$ as in 
Lemma~\ref{lem:Car} with $\delta=0$. In addition, as a consequence of \eqref{lC-1}, we have that
$$
	\frac{1}{C}\abs{\mu}^p-C\leq A_\xi(\mu)\leq C\abs{\mu}^p+C~~\text{for every $\mu \in \text{co}\{\xi_j\}$}
$$
(the convex hull of the points $\xi_j$, $j=0,\ldots,m$). Clearly, the existence of an affine function satisfying the latter implies that $\text{co}\{\xi_j\}$ is bounded for fixed $\xi$,
and it is not difficult to obtain more precise estimates that yield \eqref{lC-2}. 
\end{proof}

The following result is the crucial step towards the upper bound for $\Gamma-\limsup F_\eps$ in the general case.
\begin{prop}\label{prop:ub1}
Let $N\geq 2$, let $1\leq p<\infty$, let $I\subset (0,1)$ be an open interval and let $\eps_n\to 0^+$.
Then for every sequence $\tau_n\to 0^+$ and every pair of points $\zeta_1,\zeta_2\in \RR^N$ and numbers $\gamma_1,\gamma_2\in (0,1)$
such that $\zeta_1^N\neq \zeta_2^N$, $\gamma_1\zeta_1+\gamma_2\zeta_2=0$ and $\gamma_1+\gamma_2=1$, 
there exists a sequence $(v_n)\subset L^\infty(\RR^N;\RR^N)$
such that 
\begin{equation}\label{lub1-0}
	\norm{v_n}_{L^\infty}\leq \max\{\abs{\zeta_1},\abs{\zeta_2}\},
\end{equation}
\begin{equation}\label{lub1-2}
	\norm{\div_{\eps_n} v_n}_{W^{-1,p}(\Omega)}
	+\bignorm{\big(v'_n,\tfrac{1}{{\eps_n}}v^N_n\big)}_{W^{-1,p}(\Omega;\RR^N)}
	\leq \tau_n
\end{equation}
for every $n\in\NN$, 
\begin{equation}\label{lub1-1}
\begin{aligned}
	&v_n\rightharpoonup 0~~\text{in $L^p_\loc(\RR^N;\RR^N)$ as $n\to\infty$},~~\\
	&\supp(v_n)\subset \RR^{N-1}\times \textstyle{\bigcup_{z\in\ZZ}}
	\big({\eps_n} z+{\eps_n} I^{[{\eps_n}]} \big),
\end{aligned}
\end{equation}	
where $I^{[\eps]}:=\mysetl{t\in I}{\dist{t}{\partial I}\geq \eps}$, and 
\begin{equation}\label{lub1-1b}
\begin{aligned}
	&\measN{\{v_n=\zeta_j\}\cap U}
	\underset{n\to\infty}{\To} \gamma_j \measN{U}\abs{I}
	~\text{for every measurable set $U\subset \RR^{N}$}
\end{aligned}
\end{equation}
and $j=1,2$. 
\end{prop}
\begin{rem}
The assumption $\zeta_1^N\neq \zeta_2^N$ is actually obsolete.
The case of equality is only excluded above because it is much simpler
and will be treated separately in Proposition~\ref{prop:ub2} below.
\end{rem}
\begin{proof}[Proof of Proposition~\ref{prop:ub1}]
For each $n\in \NN$ fix a function $\varphi_n\in C_c^\infty(\RR;[0,1])$
such that $\varphi_n=1$ on $I^{[2\eps_n]}$ and $\varphi_n=0$ on $\RR\setminus I^{[\eps_n]}$, and define
$$
	\psi_n\in C^\infty(\RR;[0,1]),~~
	\psi_n:=\sum_{z\in \ZZ}\varphi_n(\cdot+z).
$$
Furthermore, for $k\in\NN$ let
$$
	 w_k(t):=\left\{
	\begin{alignedat}[c]{2}
	 	&\zeta_1~~&&\text{if $0<t\leq \gamma_1\tfrac{1}{k}$},\\
	 	&\zeta_2~~&&\text{if $-\gamma_2\tfrac{1}{k}<t\leq 0$,}
	\end{alignedat}\right.
$$
extended periodically to a function $w_k:\RR\to \RR^N$ with period $\tfrac{1}{k}$.
Note that 
\begin{align}\label{lub1-3}
	w_k\rightharpoonup \gamma_1\zeta_1+\gamma_2\zeta_2=0
	~~\text{weakly in $L^p_\loc(\RR;\RR^N)$.}
\end{align}	
With a fixed unit vector
$\zeta_{12}^\perp\in \RR^N$ perpendicular to $\zeta_1-\zeta_2$,
we define $v_{k,n}\in L^\infty(\RR^N;\RR^N)$ by 
$$
	v_{k,n}(x):=
	\psi_n(\tfrac{1}{\eps_n}x_N)
	w_k\big((\tfrac{1}{\eps_n^2}x',\tfrac{1}{\eps_n}x_N)\cdot \zeta_{12}^\perp\big)
$$
Observe that although $x\mapsto w_k\big((\tfrac{1}{\eps_n^2}x',\tfrac{1}{\eps_n}x_N)\cdot \zeta_{12}^\perp\big)$ is not continuous, it is $\div_{\eps_n}$-free (as a distribution), and thus $\div_{\eps_n} v_{k,n}$ is actually a function with 
$$
	\div_{\eps_n} v_{k,n}(x)=
	\eps_n^{-2}\dot{\psi}_n(\tfrac{2}{\eps_n}x_N)
	w_k^N\big((\tfrac{1}{\eps_n^2}x',\tfrac{1}{\eps_n}x_N)\cdot \zeta_{12}^\perp\big).
$$
In particular, as $k\to \infty$ for fixed $n$,
$\div_{\eps_n} v_{k,n}\rightharpoonup 0$ weakly in $L^p(\Omega)$ due to \eqref{lub1-3}, and thus $\div_{\eps_n} v_{k,n}\to 0$
strongly in $W^{-1,p}(\Omega)$, by compact embedding.
Analogously, we get that $v_{k,n}\to 0$ in in $W^{-1,p}(\Omega;\RR^N)$ as $k\to \infty$.
Hence, we may choose $k=k(n)$ with $k(n)\to \infty$ as $n\to \infty$ fast enough such that \eqref{lub1-2} holds for $v_n:=v_{k(n),n}$,
and \eqref{lub1-0}, \eqref{lub1-1} and \eqref{lub1-1b} hold by construction.
\end{proof}
Carathéodory's theorem requires convex combination of up to $N+1$ points,
but Proposition~\ref{prop:ub1} only admits two points.
The following elementary lemma allows us to handle general convex combinations by breaking them into suitable pairs of two. Essentially, it states that if $\xi=\sum_j \theta_j \xi_j$ is a convex combination with $\xi\in H$, where $H$ is an affine hyperplane,
then $\xi$ can be rewritten as a convex combination of points $\bar{\xi}_{ij}\in H$, such that each $\bar{\xi}_{ij}$ is a convex combination of two of the original points, i.e., $\bar{\xi}_{ij}=\beta_{ij}\xi_j+\beta_{ji}\xi_i$:
\begin{lem}\label{lem:convexa}
Let $m\leq N$ and let
$\xi_j\in \RR^N$, $\theta_j\in (0,1]$ for $j=0,\ldots,m$ such that $\sum_{j=0}^m\theta_j=1$ and the vectors $\xi_j-\xi_0$, $j=1,\ldots,m$, are linearly independent.
Then there exists numbers $\alpha_{ij}\in [0,1]$, $i,j\in \{0,\ldots,m\}$, such that
\begin{equation}\label{lca-1a}
	\alpha_{ij}=\alpha_{ji},~~\alpha_{ij}=0~\text{whenever}~\beta_{ij}=0,~~
	\sum_{j=0}^m\sum_{i=0}^{j} \alpha_{ij}=1,
\end{equation}
\begin{equation}
	\label{lca-1b}
	\theta_j=\sum_{i=0}^m \alpha_{ij}\beta_{ij},
\end{equation}
and
\begin{equation}\label{lca-2}
	\xi=\frac{1}{2}\sum_{i,j=0}^m \alpha_{ij}\big(\beta_{ij}\xi_j+\beta_{ji}\xi_i\big)
	=\sum_{i<j}\alpha_{ij}\big(\beta_{ij}\xi_j+\beta_{ji}\xi_i\big)
	+\sum_{j} \alpha_{jj}\beta_{jj}\xi_j.
\end{equation}
where
\begin{equation*}
	\beta_{ij}:=
	\left\{\begin{array}{cl}
		\frac{\xi_i^N-\xi^N}{\xi_i^N-\xi_j^N}
		&\text{if}~(\xi_i^N-\xi^N)(\xi_j^N-\xi^N)<0,\\
		1 &\text{if}~i=j~\text{and}~\xi_j^N=\xi^N,\\
		0 &\text{else}.
	\end{array}\right.
\end{equation*}
Here, note that $\beta_{ij}\in [0,1]$ and $\beta_{ij}+\beta_{ji}=1$ if $(\xi_i^N-\xi^N)(\xi_j^N-\xi^N)<0$.
\end{lem}
\begin{proof}
Let $H:=\{y\in \RR^N\mid y^N=\xi^N\}$.
Since $\xi\in S:=\co \{\xi_j\mid j=0,\ldots,m\}\cap H$ (where $\co A$ denotes the convex hull of a set $A$), which is a convex polyhedral set, $\xi$ can be written as a convex combination of the extreme points of
$S$. Such an extreme point is either given by $\xi_j$ for some $j$ such that $\xi_j^N=\xi^N$, or it is the intersection of $H$ with a line segment of the form $\co\{\xi_i,\xi_j\}$,
for indices $i,j$ such that $\xi_i$ and $\xi_j$ lie on opposite sides of $H$ (i.e., $(\xi_i^N-\xi^N)(\xi_j^N-\xi^N)<0$). Note that $\co\{\xi_j,\xi_i\}\cap H=\{\beta_{ij}\xi_j+\beta_{ji}\xi_i\}$ in this case. 
Hence, there exist $\alpha_{ij}\in [0,1]$ such that $\alpha_{ij}=\alpha_{ji}$, $\alpha_{ij}=0$ if $\beta_{ij}=0$, $\textstyle{\sum_{i\leq j} \alpha_{ij}=1}$ and
\eqref{lca-2} holds. Moreover, since $\alpha_{ij}=\alpha_{ji}$, we have that
$$
	\xi=\frac{1}{2}\sum_{i,j=0}^m \alpha_{ij}\big(\beta_{ij}\xi_j+\beta_{ji}\xi_i\big)
	=\sum_{j=0}^m \Big(\sum_{i=0}^m \alpha_{ij}\beta_{ij}\Big)\xi_j.
$$
This is another way of expressing $\xi$ as a convex combination of the points $\xi_j$. Since $\xi_j-\xi_0$, $j=1,\ldots,N$, are linearly independent, the coefficients of the convex combination are uniquely determined, and comparison yields \eqref{lca-1b}.
\end{proof}
Combining multiple instances of Proposition~\ref{prop:ub1} with Lemma~\ref{lem:convexa}, we obtain
\begin{prop}\label{prop:ub2}
Let $N\geq 2$, 
let $1\leq p<\infty$, let $J\subset (0,1)$ be an open interval and let $\eps_n\to 0^+$. Moreover, let $m\leq N$, let
$\xi_j\in \RR^N$ and $\theta_j\in (0,1]$, $j=0,\ldots,m$, be such that 
$$
	\textstyle{\sum_j} \theta_j \xi_j=0,~~\sum_{j}\theta_j=1,
$$
and the vectors $\xi_j-\xi_0$, $j=1,\ldots,m$, are linearly independent.
Then for every sequence $\sigma_n\to 0^+$,
there exist sequences $(y_n),(z_n)\subset L^\infty(\RR^N;\RR^N)$
such that 
\begin{equation}\label{lub2-0a}
	\norm{y_n}_{L^\infty}\leq \textstyle{\max_{j}} \abs{\xi_j}~~\text{and}~~
	\norm{z_n}_{L^\infty}\leq \textstyle{\max_{j}} \abs{\xi_j},
\end{equation}
\begin{equation}\label{lub2-0}
	\partial_N y_n^N=\div' y_n'=0~~\text{on $\RR^N$,}
\end{equation}
\begin{equation}\label{lub2-2}
	\norm{\div_{\eps_n} z_n}_{W^{-1,p}(\Omega)}
	+\bignorm{\big(z_n',\tfrac{1}{\eps_n}z^N_n\big)
	}_{W^{-1,p}(\Omega;\RR^N)}
	\leq \sigma_n
\end{equation}
for every $n\in\NN$,
\begin{equation}\label{lub2-1}
\begin{aligned}
	&
	y_n\rightharpoonup 0,~~z_n \rightharpoonup 0~~\text{in $L^p(\omega \times J;\RR^N)$ as $\eps\to 0^+$}0,\\	
	&\supp(y_n)\cup \supp(z_n) \subset \RR^{N-1}\times K_n
	~~\text{for a compact set $K_n\subset J$}
\end{aligned}
\end{equation}
and
\begin{equation}\label{lub2-1b}
	\measN{\{y_n+z_n=\xi_j\}\cap U}\underset{\eps\to 0^+}{\To} 
	\theta_j \measN{U},
	~~\text{for every measurable $U\subset \RR^{N-1}\times J$}
\end{equation}
and every $j\in \{0,\ldots,m\}$. 
\end{prop}
\begin{proof}
Let $\alpha_{ij}$ and $\beta_{ij}$ be as in Lemma~\ref{lem:convexa}, and divide the unit interval $(0,1)$ into pairwise disjoint open subintervals $I_{ij}$, $0\leq i\leq j\leq m$ (some possibly empty), such that $\abs{I_{ij}}=\alpha_{ij}$.
For $\eps>0$ let
\begin{align*}
	&T_{ij}(\eps):=\RR^{N-1}\times \bigcup_{k\in \ZZ} (\eps k+\eps I_{ij})
	,\qquad
	\bar{\xi}_{ij}:=\left\{\begin{array}{ll}
		\beta_{ji}\xi_i+\beta_{ij}\xi_j & \text{if $i\neq j$,}\\
		\xi_j & \text{if $i=j$.}
	\end{array}\right.
\end{align*}
For $i\leq j$, we define bounded sequence $(y_{ij,n})_n, (z_{ij,n})_n\subset L^\infty(\RR^N;\RR^N)$ as follows:
$$
	y_{ij,n}:=\chi_{T_{ij}(\eps_n)}\bar{\xi}_{ij},
$$
where $\chi_{T_{ij}(\eps_n)}$ denotes the characteristic function of the set $T_{ij}(\eps_n)$. For every $j$, we set 
$z_{jj,n}:=0$.
For $i<j$, let $z_{ij,n}$ be the sequence obtained in
Proposition~\ref{prop:ub1}, applied with $I:=I_{ij}$, $\tau_n:=\frac{1}{(m+1)(m+2)}\eps_n\sigma_n$,
$\zeta_1:=\xi_i-\bar{\xi}_{ij}$, $\zeta_2:=\xi_j-\bar{\xi}_{ij}$, 
$\gamma_1:=\beta_{ji}$ and $\gamma_2:=\beta_{ij}=1-\beta_{ji}$.
In particular, Proposition~\ref{prop:ub1} gives that
\begin{align}\label{lub2-3}
\begin{aligned}[c]
	&\measN{\{z_{ij,n}=\xi_i-\bar{\xi}_{ij}\}\cap U}\underset{n\to\infty}{\To} 
		\beta_{ji} \abs{I_{ij}} \measN{U}=\beta_{ji} \alpha_{ij} \measN{U},\\
	&\measN{\{z_{ij,n}=\xi_j-\bar{\xi}_{ij}\}\cap U}\underset{n\to\infty}{\To} 
		\beta_{ij} \abs{I_{ij}} \measN{U}=\beta_{ij} \alpha_{ij} \measN{U},\\
\end{aligned}
\end{align}
for every measurable $U\subset \RR^N$, and
\begin{align}\label{lub2-3b}
\begin{aligned}[c]
	&\supp(z_{ij,n})\subset \RR^{N-1}\times \textstyle{\bigcup_{k\in\ZZ}}
	\big(\eps_n k+\eps_n I_{ij}^{[\eps_n]} \big)
	~~\text{with a compact $I_{ij}^{[\eps_n]}\subset I_{ij}$}
\end{aligned}
\end{align}
for every $i\leq j$ (for $i=j$, \eqref{lub2-3} and \eqref{lub2-3b} are trivial).
In addition,
\begin{align}\label{lub2-4}
	\norm{\div_{\eps_n} z_{ij,n}}_{W^{-1,p}(\Omega)}
	+\norm{\big(z'_{ij,n},\tfrac{1}{\eps_n}v^N_{ij,n}\big)
	}_{W^{-1,p}(\Omega;\RR^N)}
	\leq \tfrac{1}{(m+1)(m+2)}\eps_n \sigma_n
\end{align}
for every $n$ and every $i\leq j$. 
Now let
$$
	\tilde{z}_n(x):=\sum_{j=0}^m\sum_{i=0}^{j}
	z_{ij,n}(x)
	~~\text{and}~~
	\tilde{y}_n(x):=\sum_{j=0}^m\sum_{i=0}^{j}
	y_{ij,n}(x)~~\text{for $x\in\RR^N$.}
$$ 
Note that at any given $x$, at most one term contributes in each of the double sums above; more precisely,
$\tilde{z}_n=z_{ij,n}$ and $\tilde{y}_n=y_{ij,n}$ on $T_{ij}(\eps_n)$. Moreover,
$$
	\tilde{y}_n \underset{\eps\to 0^+}{\rightharpoonup} 
	\sum_{j=0}^m\sum_{i=0}^{j}\measN{I_{ij}} \bar{\xi}_{ij}
	=\sum_{i<j} \alpha_{ij}(\beta_{ij}\xi_j+\beta_{ji}\xi_i)
	+\sum_{j}\alpha_{jj}\beta_{jj}\xi_j
	=0
$$
weakly in $L^p(\Omega;\RR^N)$,
and 
$$
	\partial_N \tilde{y}_n^N=\div' \tilde{y}_n'=0~~\text{on}~\RR^N
$$
since $\tilde{y}_n(\cdot,x_N)$ is constant for every $x_N\in\RR$ and $\tilde{y}_n^N=\bar{\xi}_{ij}^N=0$ a.e..
By \eqref{lub2-4}, we obtain that
\begin{align}\label{lub2-5}
\begin{aligned}
	&\bignorm{\div_{\eps_n} \tilde{z}_n}_{W^{-1,p}(\Omega)}
	+\bignorm{\big(\tilde{z}'_n,\tfrac{1}{\eps_n}\tilde{z}^N_n\big)
	}_{W^{-1,p}(\Omega;\RR^N)}\\
	&\qquad\qquad
	\leq \textstyle{\sum_{j=0}^m \sum_{i=0}^{j}}\tfrac{1}{(m+1)(m+2)}
	(\eps_n+1) \sigma_n=\frac{\eps_n+1}{2}\sigma_n
	\leq \sigma_n
\end{aligned}
\end{align}
for $n\in\NN$.
By \eqref{lub2-3}, we get that
\begin{align}\label{lub2-6}
\begin{aligned}[c]
	&\measN{\{\tilde{y}_n+\tilde{z}_n=\xi_j\}\cap U}\underset{n\to\infty}{\To} 
		\Big(\abs{I_{jj}}+\sum_{i\neq j}\beta_{ij} \alpha_{ij}\Big) 
		\measN{U}=\theta_j\measN{U}
\end{aligned}
\end{align}
for every $j$ and every measurable $U\subset \RR^N$,
where the latter equality is due to \eqref{lca-1b} combined with the fact that 
$\bigabs{I_{jj}^{[\eps_n]}}=\alpha_{jj}=\beta_{jj}\alpha_{jj}$.
Finally, define
$$
	z_n:=\chi_{\RR^{N-1}\times K_n} \tilde{z}_n
	~~\text{and}~~
	y_n:=\chi_{\RR^{N-1}\times K_n} \tilde{y}_n
$$
where
$$
	K_n:=\bigcup_{k\in Z_n(J)} \big(\eps_n k+\eps_n [0,1]\big)
	~~\text{and}~~
	Z_n(J):=\mysetl{k\in\ZZ}{\eps_n k+\eps_n [0,1]\subset J}.
$$
Clearly, \eqref{lub2-0a}, \eqref{lub2-1} and \eqref{lub2-1b} are satisfied, the latter as a consequence of \eqref{lub2-6}.
In addition, 
\begin{align*}
	z_n=0~~\text{and}~~y_n^N=\xi^N=0~~\text{in a vicinity of $\RR^{N-1}\times \partial K_n$},
\end{align*}
the former by \eqref{lub2-3b}.
Consequently, $\partial_N y_n^N=\partial_N \tilde{y}_n^N=0$ and $\div' y_n'=\div' \tilde{y}_n'=0$
on $\RR^N$, and \eqref{lub2-5} implies \eqref{lub2-2}.
\end{proof}
The next result essentially yields the upper bound in the piecewise constant case.
\begin{prop}\label{prop:ubpwc}
Let $f_{\#}$ be a function satisfying \eqref{f0}--\eqref{f2} and let $u_{\#}\in \cU_0$. Moreover,
let $J_k\subset (0,1)$ be a finite number of pairwise disjoint open intervals covering $(0,1)$ up to a set of measure zero, let $\omega_h\subset \omega$ be a finite number of open, pairwise disjoint sets covering $\omega$ up to a set of measure zero, and suppose that for each $(h,k)$ and each $\mu\in\RR^N$,
$$
	\text{$u_{\#}$ and $f_{\#}(\cdot,\mu)$ are constant on $Q_{h,k}$, where}~Q_{h,k}:=\omega_h\times J_k.
$$
Then for every pair of sequences $\eps_n\to 0^+$ and $\tau_n\to 0^+$, 
there exist two sequences 
$(v_n),(w_n)\subset L^\infty(\Omega;\RR^N)$
such that $v_n\rightharpoonup 0$ and $w_n \rightharpoonup 0$ in $L^p(\Omega;\RR^N)$,
\begin{equation}\label{lubpw-c1}
	\abs{v_n(x)}\leq K \big(\abs{u_{\#}(x)}+1\big)~~\text{and}~~
	\abs{w_n(x)}\leq K \big(\abs{u_{\#}(x)}+1\big)~~\text{for a.e.~$x\in \Omega$},
\end{equation}
where $K$ is a constant that only depends on the constants in \eqref{f1} and \eqref{f2},
\begin{equation}\label{lubpw-c2}
	\div_{\eps_n}v_n=0~~\text{on $\RR^N$},
\end{equation}
\begin{equation}\label{lubpw-c3}
	\norm{\div_{\eps_n} w_n}_{W^{-1,p}(\Omega)}
	+\bignorm{\big(w_n',\tfrac{1}{\eps_n}w^N_n\big)
	}_{W^{-1,p}(\Omega;\RR^N)}
	\leq \tau_n
\end{equation}
for every $n\in\NN$, and
\begin{equation}\label{pubpwc-1}
	\lim_{n\to\infty} \int_{\Omega} f_{\#}(x,u+v_n+w_n)\,dx =
	\int_{\Omega} f_{\#}^{**}(x,u)\,dx,
\end{equation} 
where for every $x$, $f_{\#}^{**}(x,\cdot)$ denotes the convex envelope of $f_{\#}(x,\cdot)$.
\end{prop}%
%
\begin{proof}
{\bf \underline{Step 1}:} We first show the assertion with \eqref{lubpw-c2} replaced by the condition
\begin{equation}\label{lubpw-c2s1}
	\partial_N v^N_n=0~~\text{on $\RR^N$}~~\text{and}~~
	\norm{\div' v_n'}_{L^\infty(\Omega)}\leq (\eps_n)^{-\frac{1}{2}}.
\end{equation}
Clearly, it is enough to define $v_n$ and $w_n$ on each $Q_{h,k}$ and prove the asserted properties with $Q_{h,k}$ instead of $\Omega$, as long as the restriction of $v_n$ and $w_n$ to any one $Q_{h,k}$ has compact support in this set. Hence, we consider $h$ and $k$ to be fixed below. 

Let $(\sigma_n)\subset (0,\infty)$ be a sequence with $\sigma_n\to 0^+$ (fast enough, as specified later), and define
$$
	u_{h,k}:=u_{\#}(x)~~\text{and}~~g_{h,k}(\mu):=f_{\#}(x,\mu+u_{h,k})~~\text{for $x\in Q_{h,k}$ and $\mu\in\RR^N$.}
$$
By Lemma~\ref{lem:Car} and Lemma~\ref{lem:Car2}, $0\in \RR^N$ can be written as a convex combination $0=\sum_{j=0}^m \theta_j \xi_j$ such that
$\xi_j-\xi_0$, $j=1,\ldots,m$, are linearly independent and
\begin{align}\label{pubpwc-3}
	\textstyle{\sum_{j=0}^m} \theta_j f_{\#}(x,\xi_j+u_{\#}(x)) = \textstyle{\sum_{j=0}^m} \theta_j g_{h,k}(\xi_j) = g_{h,k}^{**}(0) = f_{\#}^{**}(x,u_{\#}(x)),
\end{align}
for every $x\in Q_{h,k}$. Moreover, as a consequence of \eqref{lC-2},
\begin{align}
	\textstyle{\max_j} \abs{\xi_j}\leq K \big(\abs{u_{h,k}}+1\big),
\end{align}
with a constant $K$ only depending on the constants in \eqref{f1} and \eqref{f2}.
Proposition~\ref{prop:ub2} applied with $J=J_k$ yields two sequences $(y_n),(z_n)\subset L^\infty(\RR^N;\RR^N)$
such that $y_n\rightharpoonup 0$ and $z_n\rightharpoonup 0$ in $L^p_\loc$,
\begin{equation}\label{pubpwc-4aa}
	\abs{y_n(x)}\leq K \big(\abs{u_{h,k}}+1\big)~~\text{and}~~
	\abs{z_n(x)}\leq K \big(\abs{u_{h,k}}+1\big)~~\text{for $x\in\RR^N$},
\end{equation}
\begin{equation}\label{pubpwc-4a}
	\partial_N y_n^N=\div' y_n'=0~~\text{on $\RR^N$},
\end{equation}
\begin{equation}\label{pubpwc-4}
	\norm{\div_{\eps_n} z_n}_{W^{-1,p}(\Omega)}
	+\bignorm{\big(z_n',\tfrac{1}{\eps_n}z^N_n\big)
	}_{W^{-1,p}(\Omega;\RR^N)}
	\leq \sigma_n,
\end{equation}
\begin{equation}\label{pubpwc-5}
	\text{$y_n$ and $z_n$ vanish in a vicinity of $\RR^{N-1}\times \partial J_k$ (depending on $n$),}
\end{equation}
and 
\begin{equation}\label{pubpwc-6}
	\lim_{n\to\infty} \int_{Q_{h,k}} g_{h,k}(y_n+z_n)\,dx 
	=\bigabs{Q_{h,k}} \sum_{j=0}^m \theta_j g_{h,k}(\xi_j),
\end{equation} 
the latter due to \eqref{lub2-1b} and Lebesgue's theorem. 
Together with \eqref{pubpwc-3}, \eqref{pubpwc-6} yields that
\begin{equation}\label{pubpwc-7}
	\lim_{n\to\infty} \int_{Q_{h,k}} g_{h,k}(y_n+z_n)\,dx =
	\int_{Q_{h,k}} g^{**}(0)\,dx.
\end{equation} 
To obtain functions with compact support in $Q_{h,k}$, 
we have to cut off $y_n$ and $z_n$ near $(\partial \omega_h)\times J_k$.
For this purpose choose a sequence of functions 
$\eta_n \in C_c^\infty(\omega_h;[0,1])$ in such a way that 
$$
	\text{$\eta_n\nearrow 1$ pointwise}~~\text{and}~~
	\norm{\nabla \eta_n}_{L^\infty}\leq (\eps_n)^{-\frac{1}{2}}\frac{1}{K \big(\abs{u_{h,k}}+1\big)}
$$
Below, we identify $\eta_n$ with a function in $C^\infty(\RR^N)$ that is constant in $x_N$.
In particular, we have that
\begin{equation}\label{pubpwc-8}
	(1-\eta_n)y_n\to 0~~\text{and}~~(1-\eta_n)z_n\to 0~~\text{pointwise a.e.~on $Q_{h,k}$}.
\end{equation}
We define
$$
	v_n:=\eta_n y_n~~\text{and}~~w_n:=\eta_n z_n
$$
By construction, these functions have compact support in $Q_{h,k}$, $v_n\rightharpoonup 0$ in $L^p$ and $w_n\rightharpoonup 0$ in $L^p$, 
and \eqref{pubpwc-4aa} entails \eqref{lubpw-c1}. In addition, we have \eqref{lubpw-c2s1}, its second part since by \eqref{pubpwc-4a}, $\div' v_n'=(\nabla' \eta_n)\cdot y'_n=(\nabla \eta_n)\cdot y_n$ and thus
$$
	\norm{\div' v_n'}_{L^\infty}\leq \norm{\nabla \eta_n}_{L^\infty}\norm{y_n}_{L^\infty}\leq (\eps_n)^{-\frac{1}{2}}.
$$
By Lebesgue's theorem, \eqref{pubpwc-7} and \eqref{pubpwc-8} yield \eqref{pubpwc-1} for $Q_{h,k}$ instead of $\Omega$.
Finally,
$$
	\div_{\eps_n} (\eta_n z_n)=(\nabla' \eta_n)\cdot z'_n+\eta_n\div_{\eps_n}z_n,
$$
whence
\begin{align*}
	&\norm{\div_{\eps_n} (\eta_n z_n)}_{W^{-1,p}(Q_{h,k})}
	+\bignorm{\big((\eta_n z_n)',\tfrac{1}{\eps_n}(\eta_n z_n)^N\big)
	}_{W^{-1,p}(Q_{h,k};\RR^N)}\\
	&\qquad \leq \norm{\eta_n}_{W^{2,\infty}(\RR^N)} \Big(\norm{\div_{\eps_n} z_n}_{W^{-1,p}(\Omega)}
	+\bignorm{\big(z_n',\tfrac{1}{\eps_n}z^N_n\big)
	}_{W^{-1,p}(\Omega;\RR^N)}\Big)\\
	&\qquad \leq \norm{\eta_n}_{W^{2,\infty}(\RR^N)} \sigma_n
\end{align*}
by \eqref{pubpwc-4}. With $\sigma_n:=\tau_n (\norm{\eta_n}_{W^{2,\infty}})^{-1}$, this gives
\eqref{lubpw-c3} for $Q_{h,k}$ instead of $\Omega$.

{\bf \underline{Step 2}:}
We still have to modify $v_n$ to obtain \eqref{lubpw-c2} instead of \eqref{lubpw-c2s1}, while maintaining the other asserted properties. For $x\in \RR^N$ let
$$
	\tilde{v}_n(x):=v_n(x)-\eps_n \int_0^{x_N} \div' v_n'(x',t)\,dt
$$
for $x=(x',x_N)\in\Omega$, with $v_n$ as in the first step.
Since $\partial_N v_n=0$, we have $\div_\eps \tilde{v}_n=0$ on $\Omega$ by construction, and due to the second part of \eqref{lubpw-c2s1}, 
$$
	\norm{v_n-\tilde{v}_n}_{L^\infty}\leq (\eps_n)^{\frac{1}{2}}\to 0. 
$$
As a consequence of the latter,
\eqref{lubpw-c1}, \eqref{lubpw-c3} and \eqref{pubpwc-1} also hold for $\tilde{v}_n$ instead of $v_n$ (in case of \eqref{lubpw-c1} with a slightly larger constant). 
\end{proof}
The proof of the upper bound in the general framework relies on approximation and the following well-known property of Carathéodory functions.
\begin{prop}[Scorza-Dragoni, e.g.~see \cite{EkTe76B}] \label{prop:scorzadragoni}
Let $\Omega\subset \RR^N$ be open and bounded and let $f:\Omega\times \RR^N\to \RR$ be a Carathéodory function. Then for every $\delta>0$, there exists a compact set $\tilde{\Omega}\subset \Omega$ such that $\nabs{\Omega \setminus \tilde{\Omega}}<\delta$ and
$f$ is continuous on $\tilde{\Omega} \times \RR^N$.
\end{prop}

\begin{prop}[upper bound]\label{prop:ub}
Assume \eqref{f0}--\eqref{f2}, let $u\in \cU_0$ and let $\eps_n\to 0^+$. Then for every $\delta>0$,
there exists a sequence $(u_n)\subset \cU_{\eps_n}$ 
such that $u_n\rightharpoonup u$ in $L^p(\Omega;\RR^N)$,
and 
\begin{equation}\label{pub-1}
	\lim_{n\to\infty} \int_\Omega f(x,u_n)\,dx \leq 
	\int_\Omega f^{**}(x,u)\,dx+\delta.
\end{equation}
\end{prop}
\begin{rem}
Since \eqref{f2} yields a bound on $\norm{u_n}_{L^p}$ independent of $\delta$, a diagonalization argument similar to the one in the third step of the proof below
shows that the assertion of Prioposition~\ref{prop:ub} stays true even for $\delta=0$.
\end{rem}
\begin{proof}[Proof of Proposition~\ref{prop:ub}]
Using a series of approximations, the assertion is reduced to Proposition~\ref{prop:ub2}. 
Any expression of the form ``$A\approx B$'' below means that $A=B+e$, with an error $e$ whose modulus is controlled 
by a suitable fraction of $\delta$. 

{\bf \underline{Step 1}:} Assume that $u\in \cU_0$ is continuous in $\bar\Omega$ and $f$ is continuous on $\tilde{\Omega}\times\RR^N$,
for some compact $\tilde{\Omega}\subset \bar\Omega$. We claim that in this case, there
exists sequences $(u_n)\subset \cU_{\eps_n}$ and $(r_n)\subset L^p(\Omega;\RR^N)$ such that
$u_n\rightharpoonup 0$ and $r_n\to 0$ in $L^p$,
\begin{equation}\label{pub-1a}
	\lim_{n\to\infty} \int_{\tilde{\Omega}} f(x,u_n-r_n)\,dx=\lim_{n\to\infty} \int_{\tilde{\Omega}} f(x,u_n)\,dx \approx
	\int_{\tilde{\Omega}} f^{**}(x,u)\,dx
\end{equation}
and
\begin{equation}\label{pub-1b}
	\abs{u_n(x)-r_n(x)}\leq (2K+1)(\abs{u(x)}+1)~~\text{for a.e.~$x\in\Omega$},
\end{equation}
where $K$ is the constant in \eqref{pubpwc-4aa} (which, unlike $u_n$ and $r_n$, is independent of $\tilde{\Omega}$).

For the proof, we divide $\Omega$ into sets of the form $Q_{h,k}=\omega_h\times J_k$ and
define associated piecewise constant approximations of $u$ and $f$ as follows:
Let $u_{\#}=(u_{\#}^1,\ldots,u_{\#}^N)$ be given by
$$
	u_{\#}^j(x):=\inf\mysetl{\min\{u^j(x),0\}}{x\in Q_{h,k}}+\sup\mysetl{\max\{u^j(x),0\}}{x\in Q_{h,k}},
$$
for $j=1,\ldots,N$, whence $u_{\#}$ is a piecewise constant function in $\cU_0$ such that $\nabs{u_{\#}^j}\leq \nabs{u^j}$. 
Moreover, for $x\in Q_{h,k}$ let
$$
	f_{\#}(x,\cdot):=f(x_{h,k},\cdot)~~\text{with a fixed}~~
	x_{h,k}\in 
	 \left\{\begin{array}{ll}
		\tilde{\Omega}\cap Q_{h,k}&\text{if $\nabs{\tilde{\Omega}\cap Q_{h,k}}>0$},\\
		Q_{h,k}&\text{otherwise}.\\
	\end{array}\right.
$$
Note that $x_{h,k}$ can always be chosen in such a way that $f_{\#}$ satisfies \eqref{f1} and \eqref{f2} with the original constants.
In the following, let
$$
	S:=\overline{B_R(0)}\subset \RR^N,~~\text{with $R:=(2K+1)(\norm{\tilde{u}}_{L^\infty(\tilde{\Omega};\RR^N)}+1)$},
$$
where $K$ is the constant in \eqref{pubpwc-4aa}.
If the mesh size (the maximal side length of the boxes $Q_{h,k}$) is small enough, we get that
\begin{align}\label{pub-3}
	\max_{x\in \tilde{\Omega}} \abs{u(x)-u_{\#}(x)}\approx 0~~\text{and}~~\max_{x\in \tilde{\Omega},~\mu\in S}
	\abs{f(x,\mu)-f_{\#}(x,\mu)}\approx 0
\end{align}
by the uniform continuity of $u$ and $f$ on compact sets.
With the sequences $v_n$ and $w_n$ of Proposition~\ref{prop:ub2},
using \eqref{pub-3}, \eqref{lubpw-c1} and the uniform continuity of $f$ on $\tilde{\Omega}\times S$, we thus have that
\begin{align*}
	\int_{\tilde{\Omega}} f(x,u+v_n+w_n)\,dx
	&\approx 
	\int_{\tilde{\Omega}} f(x,u_{\#}+v_n+w_n)\,dx 
	\approx 
	\int_{\tilde{\Omega}} f_{\#}(x,u_{\#}+v_n+w_n)\,dx 
\end{align*}
uniformly in $n$. Similarly, \eqref{pub-3} and the uniform continuity of $f^{**}$ on $\tilde{\Omega}\times S$ yield that
\begin{align*}
	\int_{\tilde{\Omega}} f^{**}(x,u)\,dx 
	\approx 
	\int_{\tilde{\Omega}} f^{**}(x,u_{\#})\,dx 
	\approx 
	\int_{\tilde{\Omega}} f^{**}_{\#}(x,u_{\#})\,dx. 
\end{align*}
Together with \eqref{pubpwc-1} (for $f_{\#}$ and $u_{\#}$ instead of $f$ and $u$), this gives
\begin{align}\label{pub-4}
	\lim_{n\to\infty} \int_{\tilde{\Omega}} f(x,u+v_n+w_n)\,dx \approx \int_{\tilde{\Omega}} f^{**}(x,u)\,dx.
\end{align}
Finally, by Lemma~\ref{lem:rc1} and Lemma \ref{lem:proj} applied to $u$ and $w_n$, respectively, there exists a sequence $(r_n)\subset L^p(\Omega;\RR^N)$
such that $r_n\to 0$ in $L^p$ and $\div_{\eps_n} (u+w_n+r_n)=0$ on $\Omega$. By Lebesgue's theorem, \eqref{f0} and \eqref{f1},
we have that
\begin{align}\label{pub-5}
	\lim_{n\to\infty} \int_\Omega f(x,u+v_n+w_n)\,dx = 
	\lim_{n\to\infty} \int_\Omega f(x,u+v_n+w_n+r_n)\,dx, 
\end{align}
also using that $u\in L^p$ is fixed and $(v_n)$, $(w_n)$ are bounded in $L^\infty$.
Combining \eqref{pub-4} and \eqref{pub-5}, we infer \eqref{pub-1a} for $u_n:=u+v_n+w_n+r_n$,
and \eqref{pub-1b} is a consequence of \eqref{pubpwc-4aa} and the fact that $\abs{u_{\#}}\leq \abs{u}$ a.e.~in $\Omega$.

{\bf \underline{Step 2}:} Assume that $u\in \cU_0$ is continuous in $\bar\Omega$.\\
As a consequence of Proposition~\ref{prop:scorzadragoni}, there exists a compact subset $\tilde{\Omega}$ of $\Omega$ such that
$f$ is continuous on $\tilde{\Omega}\times \RR^N$,
and $\bigabs{\Omega\setminus \tilde{\Omega}}$ is small enough such that
\begin{equation}\label{pub-6}
	\int_{\Omega\setminus \tilde{\Omega}} \abs{f^{**}(x,u)}\,dx\approx 0
\end{equation}
and
\begin{equation}\label{pub-7}
	\sup_{v\in V} \int_{\Omega\setminus \tilde{\Omega}} \abs{f(x,v(x))}\,dx\approx 0,
\end{equation}
where $V:=\mysetr{v\in L^p(\Omega;\RR^N)}{\abs{v}\leq (2K+1)(\abs{u}+1)~\text{a.e.}}$. Here, note that 
the set $\mysetl{f(\cdot,v(\cdot))}{v\in V}\subset L^1(\Omega)$ is equiintegrable by \eqref{f1}. 
With the sequences $(u_n)\subset \cU_{\eps_n}$ and $(r_n)\subset L^p(\Omega;\RR^N)$ of Step 1,
we thus have that
$$
	\lim_{n\to\infty}\int_\Omega f(x,u_n-r_n)=\lim_{n\to\infty}\int_\Omega f(x,u_n)\approx \int_\Omega f^{**}(x,u)
$$
due to \eqref{pub-1a}, \eqref{pub-6}, \eqref{pub-1b} and \eqref{pub-7}. 

{\bf \underline{Step 3}:} The general case.\\
By Lemma~\ref{lem:U0density}, there exists a sequence $(\tilde{u}_k)\subset \cU_0\cap C(\bar{\Omega};\RR^N)$ with $\tilde{u}_k\to u$ in $L^p(\Omega)$.
Let $(\tilde{u}_{k,n})\subset \cU_{\eps_n}$ and $(\tilde{r}_{k,n})\subset L^p(\Omega;\RR^N)$ denote the sequences
corresponding to $\tilde{u}_k$ obtained in the previous step. By \eqref{pub-1b}, $\tilde{u}_{k,n}-\tilde{r}_{k,n}$ is bounded in $L^p$, uniformly in $k$ and $n$.
Since the dual of $L^p(\Omega;\RR^N)$ is separable, 
$\tilde{u}_{k,n}-\tilde{r}_{k,n}\rightharpoonup u_k$ in $L^p$ as $n\to\infty$, $\tilde{u}_k\to u$ in $L^p$ as $k\to\infty$,
$\tilde{r}_{k,n}\to 0$ in $L^p$ as $n\to\infty$, and
$$
	\lim_{n\to\infty}\int_\Omega f(x,\tilde{u}_{k,n})\approx \int_\Omega f^{**}(x,\tilde{u}_k)\underset{k\to\infty}{\To}\int_\Omega f^{**}(x,u),
$$
there exist diagonal sequences 
$$
	u_n:=\tilde{u}_{k(n),n}\in\cU_{\eps_n}~~\text{and}~~r_n:=\tilde{r}_{k(n),n}\in L^p(\Omega;\RR^N)
$$
with $k(n)\to \infty$ slow enough such that
$u_n-r_n\rightharpoonup u$ in $L^p$, 
$r_n\to 0$ in $L^p$,
and
$$
	\lim_{n\to\infty}\int_\Omega f(x,u_n)\approx \int_\Omega f^{**}(x,u).
$$

\vspace*{-6ex}
\end{proof}
\begin{rem}\label{rem:matrixcase}
It is natural to ask whether our result also holds for functionals on $\Div$-free matrix fields (i.e., each column is divergence-free).
The approach presented here extends in a straightforward way to fields with values in $\RR^{N\times M}$ for $M\leq N-1$, 
but it does not work for $M\geq N$. Of course, for $M\geq N$, the matrices can have rank $N$, and 
in general, it is no longer clear if
$\Div$-quasiconvexity (S-quasiconvexity in the terminogy of \cite{Pa09ap}, which implies convexity along directions of $\rank\leq N-1$) implies convexity. 
We expect that in this case, the convex envelope in Theorem~\ref{thm:main} has to be replaced by a 
suitable variant of a quasiconvex envelope. We hope to address this in a future work.
\end{rem}

\subsection*{Acknowledgements}
I am grateful to Irene Fonseca, who suggested this topic, and to Martin Kruzik, for useful discussions and remarks on the subject.
Part of the research was carried out during a stay at Carnegie Mellon University in Pittsburgh, made possible
by the financial support of the Deutsche Forschungsgemeinschaft (fellowship KR 3544/1-2) 
and the hospitality of the Center for Nonlinear Analysis (NSF Grants No. DMS-0405343 and DMS-0635983).
\bibliographystyle{plain}
\bibliography{10_bib}
\end{document}

%% file: 10macros.tex
\usepackage[T1]{fontenc}
\usepackage[ansinew]{inputenc}
\usepackage{amsmath}                       
\usepackage{amssymb}                       
\usepackage{amsfonts}                      
\usepackage{amsthm}     
\usepackage{enumerate}                     

\newcommand{\TTT}{\mathbb T}
\newcommand{\RR}{\mathbb R}
\newcommand{\NN}{\mathbb N}
\newcommand{\ZZ}{\mathbb Z}

\newcommand{\cA}{\mathcal{A}}

\newcommand{\cV}{\mathcal{V}}

\newcommand{\cP}{\mathcal{P}}

\newcommand{\cU}{\mathcal{U}}
\newcommand{\eps}{\varepsilon}

\newcommand{\abs}[1]{\left\vert#1\right\vert}

\newcommand{\nabs}[1]{\vert#1\vert}
\newcommand{\bigabs}[1]{\big\vert#1\big\vert}

\newcommand{\measN}[1]{\abs{#1}}

\newcommand{\norm}[1]{\left\|#1\right\|}
\newcommand{\bignorm}[1]{\big\|#1\big\|}
\newcommand{\nnorm}[1]{\|#1\|}
\newcommand{\dist}[2] {\operatorname{dist}\left(#1;#2\right)}

\newcommand{\mysetr}[2] {\left\{#1\,\left|\,#2\right.\right\}}
\newcommand{\mysetl}[2] {\left\{\left.#1\,\right|\,#2\right\}}
\newcommand{\bigset}[2] {\big\{#1\,\big|\,#2\big\}}

\newcommand{\loc}{\text{loc}}

\newcommand{\supp}{\operatorname{supp}}
\renewcommand{\div}{\operatorname{div}}
\newcommand{\Div}{\operatorname{Div}}

\newcommand{\co}{\operatorname{co}}

\newcommand{\To}{\longrightarrow}

\newcommand{\rank}{\operatorname{rank}}

\renewcommand{\labelenumi}{(\roman{enumi})}
\newenvironment{myenum}{\vspace*{-1ex}\begin{enumerate}}{\end{enumerate}\vspace*{-1ex}}

\newtheorem{theorem}{Theorem}

\newtheorem{thm}[theorem]{Theorem}

\newtheorem{lem}[theorem]{Lemma}
\newtheorem{prop}[theorem]{Proposition}
\theoremstyle{definition}

\newtheorem{ex}[theorem]{Example}
\theoremstyle{remark}
\newtheorem{rem}[theorem]{Remark}
\renewenvironment{proof}[1][\proofname]{\par
  \normalfont
  \trivlist
  \item[\hskip\labelsep\itshape
    \bfseries{#1.}]\ignorespaces
}{%
  \qed\endtrivlist
}

%% file: div-free_3d-2d.bbl
\begin{thebibliography}{10}

\bibitem{AlBroGa10a}
A.~Alama, L.~Bronsard, and B.~Galv{\~a}o-Sousa.
\newblock Thin film limits for {Ginzburg-Landau} for strong applied magnetic
  fields.
\newblock {\em SIAM J. Math. Anal.}, 42(1):97--124, 2010.

\bibitem{AnsGarr07a}
Nadia Ansini and Adriana Garroni.
\newblock {$\Gamma$-convergence of functionals on divergence-free fields}.
\newblock {\em ESAIM, Control Optim. Calc. Var.}, 13(4):809--828, 2007.

\bibitem{Ba89a}
J.~M. Ball.
\newblock A version of the fundamental theorem for young measures.
\newblock In M.~Rascle, D.~Serre, and M.~Slemrod, editors, {\em PDEs and
  continuum models of phase transitions. Proceedings of an NSF-CNRS joint
  seminar held in Nice, France, January 18-22, 1988}, volume 344 of {\em
  Lect.~Notes Phys.}, pages 207--215, Berlin etc., 1989. Springer.

\bibitem{Brai02a}
Andrea Braides.
\newblock {\em {$\Gamma$-convergence} for beginners}, volume~22 of {\em Oxford
  Lecture Series in Mathematics and its Applications}.
\newblock Oxford University Press, Oxford, 2002.

\bibitem{BraiFoLe00a}
Andrea Braides, Irene Fonseca, and Giovanni Leoni.
\newblock {A-quasiconvexity: Relaxation and homogenization}.
\newblock {\em ESAIM, Control Optim. Calc. Var.}, 5:539--577, 2000.

\bibitem{CoSte10a}
A.~Contreras and P.~Sternberg.
\newblock Gamma-convergence and the emergence of vortices for {Ginzburg–Landau}
  on thin shells and manifolds.
\newblock {\em Calc. Var. Partial Differ. Equ. (to appear)}.

\bibitem{DaFoLe09ap}
G.~Dal~Maso, I.~Fonseca, and G.~Leoni.
\newblock Nonlocal character of the reduced theory of thin films with higher
  order perturbations.
\newblock Preprint 09-CNA-011.

\bibitem{Dal93B}
Gianni Dal~Maso.
\newblock {\em {An introduction to $\Gamma$-convergence.}}
\newblock Number~8 in {Progress in Nonlinear Differential Equations and their
  Applications}. {Birkh\"auser}, Basel, 1993.

\bibitem{DeGioDalMa83a}
E.~De~Giorgi and G.~Dal~Maso.
\newblock Gamma-convergence and calculus of variations.
\newblock In {\em Mathematical theories of optimization, Proc. Conf., Genova
  1981, Lect. Notes Math. 979}, pages 121--143, 1983.

\bibitem{DeGioFra75a}
E.~De~Giorgi and T.~Franzoni.
\newblock {Su un tipo di convergenza variazionale.}
\newblock {\em Atti Accad. Naz. Lincei, VIII. Ser., Rend., Cl. Sci. Fis. Mat.
  Nat.}, 58:842--850, 1975.

\bibitem{EkTe76B}
Ivar Ekeland and Roger Temam.
\newblock {\em Convex analysis and variational problems}, volume~1 of {\em
  Studies in Mathematics and its Applications}.
\newblock North-Holland Publishing Company, Amsterdam, Oxford, 1976.

\bibitem{FoKroe10a}
I.~Fonseca and S.~Kr{\"{o}}mer.
\newblock Multiple integrals under differential constraints: two-scale
  convergence and homogenization.
\newblock {\em Indiana Univ. Math. J. (to appear). Preprint 09-CNA-018.}

\bibitem{FoLe07B}
Irene Fonseca and Giovanni Leoni.
\newblock {\em {Modern methods in the calculus of variations. $L^p$ spaces.}}
\newblock {Springer Monographs in Mathematics. New York, NY: Springer}, 2007.

\bibitem{FoMue99a}
Irene Fonseca and Stefan {M\"uller}.
\newblock {$\cal A$-quasiconvexity, lower semicontinuity, and Young measures}.
\newblock {\em SIAM J. Math. Anal.}, 30(6):1355--1390, 1999.

\bibitem{FrieJaMue06a}
G.~Friesecke, R.~D. James, and S.~M{\"u}ller.
\newblock A hierarchy of plate models derived from nonlinear elasticity by
  gamma-convergence.
\newblock {\em Arch. Ration. Mech. Anal.}, 180(2):183--236, 2006.

\bibitem{GaNe04a}
A.~Garroni and V.~Nesi.
\newblock Rigidity and lack of rigidity for solenoidal matrix fields.
\newblock {\em Proc. R. Soc. Lond., Ser. A, Math. Phys. Eng. Sci.},
  460(2046):1789--1806, 2004.

\bibitem{Giu03B}
Enrico Giusti.
\newblock {\em Direct methods in the calculus of variations}.
\newblock World Scientific, Singapore, 2003.

\bibitem{LeDreRa95a}
{Herv\'e} Le~Dret and Annie Raoult.
\newblock {The nonlinear membrane model as variational limit of nonlinear
  three-dimensional elasticity.}
\newblock {\em J. Math. Pures Appl., IX. {S\'er.}}, 74(6):549--578, 1995.

\bibitem{LeDreRa96a}
{Herv\'e} Le~Dret and Annie Raoult.
\newblock {The membrane shell model in nonlinear elasticity: A variational
  asymptotic derivation.}
\newblock {\em J. Nonlinear Sci.}, 6(1):59--84, 1996.

\bibitem{LeDreRa00a}
{Herv\'e} Le~Dret and Annie Raoult.
\newblock Variational convergence for nonlinear shell models with directors and
  related semicontinuity and relaxation results.
\newblock {\em Arch. Ration. Mech. Anal.}, 154(2):101--134, 2000.

\bibitem{LeeMueMue08a}
J.~Lee, P.F.X. M{\"u}ller, and S.~M{\"u}ller.
\newblock Compensated compactness, separately convex functions and
  interpolatory estimates between {Riesz} transforms and {Haar} projections.
\newblock Preprint MPI-MIS 7/2008.

\bibitem{LeMaPa10a}
M.~Lewicka, L.~Mahadevan, and R.~Pakzad.
\newblock The {Von Kármán} equations for plates with residual strain.
\newblock {\em Preprint 10-CNA-002}, 2010.

\bibitem{LePa09a}
M.~Lewicka and R.~Pakzad.
\newblock The infinite hierarchy of elastic shell models: some recent results
  and a conjecture.
\newblock {\em Fields Institute Communications (to appear)}.

\bibitem{Mue99c}
Stefan M{\"u}ller.
\newblock Rank-one convexity implies quasiconvexity on diagonal matrices.
\newblock {\em Int. Math. Res. Not.}, 1999(20):1087--1095, 1999.

\bibitem{Mue99a}
Stefan M{\"{u}}ller.
\newblock Variational models for microstructure and phase transisions.
\newblock In S.~Hildebrandt, editor, {\em Calculus of variations and geometric
  evolution problems. Lectures given at the 2nd session of the Centro
  Internazionale Matematico Estivo (CIME), Cetraro, Italy, June 15-22, 1996},
  volume 1713 of {\em Lect.~Notes Math.}, pages 85--210, Berlin, 1999.
  Springer.

\bibitem{Mu81a}
Fran{\c{c}}ois Murat.
\newblock {Compacit\'{e} par compensation: condition necessaire et suffisante
  de continuit\'{e} faible sous une hypoth\`ese de rang constant}.
\newblock {\em Ann. Sc. Norm. Super. Pisa, Cl. Sci., IV. Ser.}, 8:69--102,
  1981.

\bibitem{PaSmy09a}
M.~Palombaro and V.P. Smyshlyaev.
\newblock Relaxation of three solenoidal wells and characterization of extremal
  three-phase {$H$}-measures.
\newblock {\em Arch. Ration. Mech. Anal.}, 194(3):775--822, 2009.

\bibitem{Pa09ap}
Mariapia Palombaro.
\newblock On the relationship between {rank-$(n-1)$} convexity and {${\mathcal
  S}$}-quasiconvexity.
\newblock Preprint arXiv:0904.4190, 2009.

\bibitem{Pe97B}
Pablo Pedregal.
\newblock {\em Parametrized measures and variational principles}, volume~30 of
  {\em Progress in Nonlinear Differential Equations and their Applications}.
\newblock Birkh{\"{a}}user, Basel, 1997.

\bibitem{Ro70B}
R.~Tyrrell Rockafellar.
\newblock {\em Convex analysis}.
\newblock Princeton University Press, Princeton, N. J., 1970.

\bibitem{Ta79a}
Luc Tartar.
\newblock Compensated compactness and applications to partial differential
  equations.
\newblock {Nonlinear analysis and mechanics: Heriot-Watt Symp., Vol. 4,
  Edinburgh 1979, Res. Notes Math. 39, 136-212}, 1979.

\end{thebibliography}
